\pdfoutput=1
\documentclass[letterpaper,    
               english,        
               11pt,           
               twoside]        
               {amsart}        
\usepackage[english]{babel}                   
\usepackage[letterpaper,                      
            bookmarks,                        
            bookmarksopen=true,               
            bookmarksnumbered=true,           
            pdfauthor={Christof Puhle},       
            pdftitle={Spin(7)-manifolds       
                      with parallel
                      torsion form},
            pdfkeywords={Spin(7)-manifolds,
                         connections with torsion,
                         parallel spinors},
            colorlinks,                       
            linkcolor=black,                  
            citecolor=black,                  
            urlcolor=black,                   
            plainpages=false]                 
            {hyperref}                        
\usepackage[nobysame]
           {amsrefs}                          
\usepackage{amssymb}                          
\usepackage{bm}                               
\usepackage{color}                            
\advance\oddsidemargin by -0.93cm
\advance\evensidemargin by -0.92cm
\advance\topmargin by -0.55cm
\theoremstyle{plain}

\newtheorem{lem}{Lemma}[section]
\newtheorem{thm}{Theorem}[section]            
\newtheorem{prop}{Proposition}[section]
\theoremstyle{definition}
\newtheorem{exa}{Example}[section]
\newtheorem{rmk}{Remark}[section]

\newtheorem*{thank}{Thanks}
\newcommand{\be}[1][*]{\begin{equation#1}}
\newcommand{\ee}[1][*]{\end{equation#1}}
\newcommand{\barr}[1]{\begin{array}{#1}}
\newcommand{\earr}{\end{array}}
\newcommand{\btab}[1]{\renewcommand{\arraystretch}{1.2}\begin{center}\begin{tabular}{#1}}
\newcommand{\etab}{\end{tabular}\end{center}\renewcommand{\arraystretch}{1.0}}
\newcommand{\bnum}{\begin{enumerate}}
\newcommand{\enum}{\end{enumerate}}
\newcommand{\bcen}{\begin{center}}
\newcommand{\ecen}{\end{center}}
\newcommand{\sst}{\scriptstyle}
\newcommand{\mca}[1]{\mathcal{#1}}
\newcommand{\mrm}[1]{\mathrm{#1}}
\newcommand{\mfr}[1]{\mathfrak{#1}}
\def\side#1{\ifvmode\leavevmode\fi\vadjust{\vbox to0pt{\vss
\hbox to 0pt{\hskip\hsize\hskip1em                     
\vbox{\hsize2cm\small\raggedright\pretolerance10000       
\noindent\textcolor{red}{#1}\hfill}\hss}\vbox to8pt{\vfil}\vss}}}
\def\hook{\mathbin{
                 \hskip-0.2pt                
                 \vrule height0.4pt width5.0pt depth0pt   
                 \kern-0.4pt
                 \vrule height6.0pt width0.4pt depth0pt
                 \hskip0.2pt  }}
\def\shook{\mathbin{
                 \hskip-0.2pt                  
                 \vrule height0.3pt width3.8pt depth0pt   
                 \kern-0.3pt
                 \vrule height4.7pt width0.3pt depth0pt
                 \hskip0.2pt }}                         
\newcommand{\x}{\times}
\newcommand{\op}{\oplus}
\newcommand{\ox}{\otimes}

\newcommand{\ra}{\rightarrow}

\newcommand{\Ra}{\Rightarrow}

\newcommand{\all}{\forall}

\newcommand{\R}{\ensuremath{\mathbb{R}}}
\newcommand{\C}{\ensuremath{\mathbb{C}}}

\newcommand{\Pro}{\ensuremath{\mathbb{P}}}
\newcommand{\Flag}{\ensuremath{\mathbb{F}}}

\newcommand{\W}{\ensuremath{\mathcal{W}}}

\newcommand{\T}{\ensuremath{\mathrm{T}}}
\newcommand{\G}{\ensuremath{\mathrm{G}}}

\newcommand{\vphi}{\ensuremath{\varphi}}

\newcommand{\diag}{\ensuremath{\mathrm{diag}}}
\newcommand{\Id}{\ensuremath{\mathrm{Id}}}

\newcommand{\Ric}{\ensuremath{\mathrm{Ric}}}
\newcommand{\Scal}{\ensuremath{\mathrm{Scal}}}
\newcommand{\hol}{\ensuremath{\mathfrak{hol}}}
\newcommand{\Hol}{\ensuremath{\mathrm{Hol}}}
\newcommand{\iso}{\ensuremath{\mathfrak{iso}}}

\newcommand{\un}{\ensuremath{\mathfrak{u}}}
\newcommand{\Un}{\ensuremath{\mathrm{U}}}
\newcommand{\su}{\ensuremath{\mathfrak{su}}}
\newcommand{\SU}{\ensuremath{\mathrm{SU}}}

\newcommand{\Spl}{\ensuremath{\mathrm{Sp}}}

\newcommand{\so}{\ensuremath{\mathfrak{so}}}
\newcommand{\SO}{\ensuremath{\mathrm{SO}}}
\newcommand{\spin}{\ensuremath{\mathfrak{spin}}}
\newcommand{\Spin}{\ensuremath{\mathrm{Spin}}}
\newcommand{\g}{\ensuremath{\mathfrak{g}}}
\newcommand{\h}{\ensuremath{\mathfrak{h}}}

\begin{document}
\thispagestyle{empty}
\date{\today}
\title{Spin(7)-manifolds with parallel torsion form}
\author{Christof Puhle}
\address{
Institut f\"ur Mathematik \newline\indent
Humboldt-Universit\"at zu Berlin\newline\indent
Unter den Linden 6\newline\indent
10099 Berlin, Germany}
\email{\noindent puhle@math.hu-berlin.de}
\urladdr{www.math.hu-berlin.de/~puhle}
\thanks{Supported by the SFB 647: `Space--Time--Matter', German Research Foundation DFG}
\subjclass[2000]{Primary 53C25; Secondary 81T30}
\keywords{$\Spin\left(7\right)$-manifolds, connections with torsion, parallel spinors}
\begin{abstract}
Any $\Spin\left(7\right)$-manifold admits a metric connection $\nabla^c$ with totally skew-symmetric
torsion $\T^c$ preserving the underlying structure. We classify those with $\nabla^c$-parallel $\T^c
\neq0$ and non-Abelian isotropy algebra $\iso\left(\T^c\right)\leqslant\spin\left(7\right)$. These are
isometric to either Riemannian products or homogeneous naturally reductive spa\-ces, each admitting
two $\nabla^c$-parallel spinor fields.
\end{abstract}
\maketitle
\setcounter{tocdepth}{1}
\bcen
\begin{minipage}{0.7\linewidth}
    \begin{small}
      \tableofcontents
    \end{small}
\end{minipage}
\ecen
\pagestyle{headings}
%
%
%
%
\section{Introduction}\noindent
In the early $`80$'s physicists tried to incorporate torsion into superstring and supergravity
theories in order to get a physically flexible model. Strominger described the mathematics of
the underlying superstring theory of type II. It consists of a Riemannian spin manifold $\left(M^n,
g\right)$ equipped, amongst other things, with a spinor field $\Psi$ and a $3$-form $\T$ that
satisfy a certain set of field equations (see \cite{Str86}), including
\be
\nabla^{g}_X\Psi+\frac{1}{4}\,\left(X\hook\T\right)\cdot\Psi=0\quad\all\,X\in TM^n.
\ee
We denote by $\tau\cdot\Psi$ the Clifford product of a differential form $\tau$ with a spinor
field $\Psi$. In the theory $\T$ is seen as a field strength of sorts, whilst $\Psi$ is the
so-called supersymmetry. With the metric connection $\nabla$ whose torsion is the $3$-form $\T$,
\be
g\left(\nabla_X Y,Z\right) = g\left(\nabla^g_X Y,Z\right) + \frac{1}{2}\cdot\T\left(X,Y,Z\right)\quad\all\,X,Y,Z\in TM^n,
\ee
the above equation transforms to $\nabla\Psi=0$. In other words, the spinor field $\Psi$ is parallel
with respect to $\nabla$, a fact imposing restrictions on the holonomy group $\Hol\left(\nabla\right)$.
In the case $\T=0$, i.e.\ when $\nabla$ is the Levi-Civita connection of $\left(M^n,g\right)$, the
holonomy group is one of the following (see \cite{McKW89}):
\be
\SU\left(n\right),\quad \Spl\left(n\right),\quad \G_2,\quad \Spin\left(7\right).
\ee
In order to construct models with $\T\neq0$, it is therefore reasonable to study manifolds that admit
a metric connection $\nabla$ with totally skew-symmetric torsion whose $\Hol\left(\nabla\right)$ is contained
in $\SU\left(n\right)$, $\Spl\left(n\right)$, $\G_2$ or $\Spin\left(7\right)$. Surprisingly, the existence
of such a connection is unobstructed for $\Spin\left(7\right)$-manifolds $M^8$ (see \cite{Iva04}). Furthermore
this connection, denoted by $\nabla^c$, is unique, preserves the underlying $\Spin\left(7\right)$-structure
and makes a non-trivial spinor field parallel. The more general point of view of \cites{FI02,AFNP05,Puh07}
indicates that structures with parallel torsion form $\T^c$,
\be
\nabla^c\T^c=0,
\ee
are of particular interest and provide a starting point to solve the entire system of Strominger's equations.
For example, $\delta\T^c=0$ is automatically satisfied in this setup. Moreover, many geometric properties
become algebraically tractable by assuming the parallelism of $\T^c$, for this implies, for instance, that
the holonomy algebra $\hol\left(\nabla^c\right)$ becomes a subalgebra of the isotropy algebra $\iso\left(\T^c
\right)\leqslant\spin\left(7\right)$.

The aim of the paper is the classification of $\Spin\left(7\right)$-manifolds with parallel torsion form $\T^c
\neq0$ and non-Abelian $\iso\left(\T^c\right)\leqslant\spin\left(7\right)$. We show that the latter fall into
eight types. For all of these algebrae we describe the admissible torsion forms $\T^c$ and Ricci tensors $\Ric^c$
with respect to $\nabla^c$. Finally we discuss the geometry of the space $M^8$ relatively to its holonomy algebra
$\hol\left(\nabla^c\right)\leqslant\iso\left(\T^c\right)$ and to the $\Spin\left(7\right)$-orbit of the torsion
form
\be
\T^c\in\Lambda^3=\Lambda^3_8\op\Lambda^3_{48}.
\ee
The main result is that these spaces are isometric to either a Riemannian product or a homogeneous naturally
reductive space; some of them are uniquely determined (see theorem \ref{thm:2} and theorem \ref{thm:3}).
Moreover, every structure admits at least two $\nabla^c$-parallel spinor fields. There are examples exhibiting
$16$ $\nabla^c$-parallel spinor fields and satisfying the additional constraint
\be
\Ric^c=\Ric^g_{ij}-\frac{1}{4}\,\T^c_{imn}\T^c_{jmn}=0
\ee
for the energy-momentum tensor (see examples \ref{exa:1} and \ref{exa:2}).

The paper is structured as follows: In \autoref{sec:2} we state basic facts on metric connections with parallel,
totally skew-symmetric torsion. We then specialize to the case of $\Spin\left(7\right)$-structures in \autoref{sec:3}.
Section \ref{sec:4} is devoted to the study of the non-Abelian subalgebrae of $\spin\left(7\right)$ used for the
algebraic classification (see \autoref{sec:5}) in terms of the torsion form. In the last section we discuss the
geometry of each of these classes.
%
%
%
%
\section{Parallel torsion}\label{sec:2}\noindent
Fix a Riemannian spin manifold $\left(M^n,g\right)$, a $3$-form $\T$, and denote the Levi-Civita connection
by $\nabla^g$. The equation
\be
g\left(\nabla_X Y,Z\right) = g\left(\nabla^g_X Y,Z\right) + \frac{1}{2}\cdot\T\left(X,Y,Z\right)\quad\all\,X,Y,Z\in TM^n,
\ee
defines a metric connection $\nabla$ with totally skew-symmetric torsion $\T$. We will consider the case of
parallel torsion, $\nabla \T=0$. Then the $3$-form $\T$ is coclosed (see \cite{FI02}), $\delta \T=0$, its
differential is given by
\be
d\T=\sum_i\left(e_i\hook \T\right)\wedge\left(e_i\hook \T\right)=:2\,\sigma^{\T}
\ee
for a chosen orthonormal frame $\left(e_1,\ldots,e_n\right)$, and the curvature tensor $\mrm{R}^\nabla$ of $\nabla$
is a field of symmetric endomorphism of $\Lambda^2$. If there exists a $\nabla$-parallel spinor
field $\Psi$ one can compute the Ricci tensor $\Ric^\nabla$ of $\nabla$ algebraically (see for example \cite{AFNP05}),
\be
2\,\Ric^\nabla\left(X\right)\cdot\Psi = \left(X\hook d\T\right)\cdot\Psi.
\ee
Moreover, the following relation holds (see \cite{AF04}):
\be
4\,\T^2\cdot\Psi := 4\,\T\cdot\left(\T\cdot\Psi\right) = \left(2\,\Scal^g+\|\T\|^2\right)\cdot\Psi.
\ee
Consequently, $\T^2$ acts as a scalar on the space of $\nabla$-parallel spinor fields, hence it
gives an algebraic restriction on $\T$.
%
%
%
%
\section{\texorpdfstring{$\Spin\left(7\right)$}{Spin(7)}-manifolds}\label{sec:3}\noindent
Consider the space $\R^8$, fix an orientation and denote a chosen oriented orthonormal
basis by $\left(e_1,\ldots,e_8\right)$. The compact simply connected Lie group $\Spin\left(7\right)$ can be
described (see for example \cite{HL82}) as the isotropy group of the $4$-form $\Phi$,
\be\label{eqn:1}
\Phi =  \phi+\ast\phi, \tag{$\star$}
\ee
where $\phi$, $Z$ and $D$ denote the following forms:
\be
\phi :=\left(Z \wedge e_7 + D \right)\wedge e_{8},\quad Z := e_{12}+e_{34}+e_{56},\quad D := e_{246}-e_{235}-e_{145}-e_{136}.
\ee
Here and henceforth we shall not distinguish between vectors and covectors and use the
notation $e_{i_1\ldots i_m}$ for the exterior product $e_{i_1}\wedge\ldots\wedge e_{i_m}$.
The so-called \emph{fundamental form} $\Phi$ is self-dual with respect to the Hodge star
operator, $\ast\Phi=\Phi$, and the $8$-form $\Phi\wedge\Phi$ is a non-zero multiple of the
volume form of $\R^8$.

A \emph{$\Spin\left(7\right)$-structure/manifold} is a triple $\left(M^8,g,\Phi\right)$ consisting of an
$8$-di\-men\-sio\-nal Riemannian manifold $\left(M^8,g\right)$ and a $4$-form $\Phi$ such that there exists
an oriented orthonormal \emph{adapted frame} $\left(e_1,\ldots,e_8\right)$ realizing (\ref{eqn:1}) at every
point. Equivalently, these structures can be defined as a reduction of the structure group of orthonormal
frames of the tangent bundle to $\Spin\left(7\right)$. The space of $3$-forms decomposes into two irreducible
$\Spin\left(7\right)$-modules,
\be
\Lambda^3=\Lambda^3_8\op\Lambda^3_{48},
\ee
which can be characterized using the fundamental form as
\be
\Lambda^3_8:=\left\{\ast\left(\beta\wedge\Phi\right)\,\,:\,\,\beta\in\Lambda^1\right\},\quad
\Lambda^3_{48}:=\left\{\gamma\in\Lambda^3\,\,:\,\,\gamma\wedge\Phi=0\right\}.
\ee
The subscript specifies the dimension of the respective space. We will denote the
projection of a $3$-form $\T$ onto one of these spaces by $\T_8$ or $\T_{48}$ respectively.

Any $\Spin\left(7\right)$-manifold admits (see \cite{Iva04}) a unique metric connection $\nabla^c$
(the \emph{characteristic connection}) with totally skew-symmetric torsion $\T^c$ (the
\emph{characteristic torsion}) preserving the $\Spin\left(7\right)$-structure, $\nabla^c\Phi=0$,
and $\T^c$ is given by
\be
\T^c=-\delta\Phi-\frac{7}{6}\,\ast\left(\theta\wedge\Phi\right).
\ee
Here $\theta\in\Lambda^1$ denotes the so-called \emph{Lee form}
\be
\theta := \frac{1}{7}\,\ast\left(\delta\Phi\wedge\Phi\right) = \frac{6}{7}\,\ast\left(\Phi\wedge \T^c\right) =
        -\frac{1}{7}\,\ast\left(\ast d\Phi\wedge\Phi\right).
\ee
The Riemannian scalar curvature $\Scal^g$ and the scalar curvature $\Scal^c$ of $\nabla^c$ are given by
(see \cite{Iva04})
\begin{align*}\label{eqn:4}
\Scal^g&=\frac{49}{18}\,\|\theta\|^2-\frac{1}{2}\,\|\T^c\|^2+\frac{7}{2}\,\delta\theta, &
\Scal^c&=\Scal^g-\frac{3}{2}\,\|\T^c\|^2. \tag{$\diamond$}
\end{align*}
Analyzing the algebraic type of $\T^c$ we obtain Cabrera's description \cite{Cab95} -- by differential
equations involving the Lee form -- of the Fern\'andez classification \cite{Fer86} of $\Spin\left(7
\right)$-structures. For example, a $\Spin\left(7\right)$-structure is of class $\W_1$, i.e.\ a
\emph{balanced} structure, if and only if the Lee form vanishes. Equivalently, these structures can
be characterized by $\T^c_8=0$. $\Spin\left(7\right)$-structures of class $\W_0$ -- the so-called
\emph{parallel} structures -- are defined by a closed fundamental form, $d\Phi=0$. These are the
structures with vanishing torsion, $\T^c=0$. In \cite{Cab95} Cabrera shows that the Lee form of a
$\Spin\left(7\right)$-structure of class $\W_2$ (for which $d\Phi=\theta\wedge\Phi$ holds or, equivalently,
$\T^c_{48}=0$) is closed, and therefore such a manifold is locally conformally equivalent to a parallel
$\Spin\left(7\right)$-manifold. These are called \emph{locally conformal parallel}. Finally, structures
of class $\W=\W_1+\W_2$ are characterized by $\T^c_8\neq0$ and $\T^c_{48}\neq0$. We summarize the previous
facts in the following table:
\btab{|c|c|c|c}
\cline{1-3}
class & characteristic torsion & differential equations&\\
\cline{1-3}
$\W_0$ (parallel) & $\T^c_8=0$, $\T^c_{48}=0$ & $d\Phi=0$, $\theta=0$&\\
\cline{1-3}
$\W_2$ (locally conformal parallel) & $\T^c_{48}=0$ & $d\Phi=\theta\wedge\Phi$&\\
\cline{1-3}
$\W_1$ (balanced) & $\T^c_8=0$ & $\theta=0$&\\
\cline{1-3}
$\W=\W_1+\W_2$ & $\T^c_8\neq0$, $\T^c_{48}\neq0$ & ---&$\!\!.$\\
\cline{1-3}
\etab

We now restrict to parallel characteristic torsion, $\nabla^c\T^c=0$.
\begin{lem}
The following formulae hold in presence of parallel characteristic torsion:
\begin{align*}
\|\theta\|^2&=\frac{36}{7}\,\|\T^c_8\|^2, & \delta\theta&=0.
\end{align*}
\end{lem}
\begin{proof}
We prove the second equation,
\be
\delta\theta = -\ast d\ast\theta=-\frac{6}{7}\,\ast d\ast\ast\left(\Phi\wedge \T^c\right)=
\frac{6}{7}\,\ast d\left(\Phi\wedge \T^c\right).
\ee
The $7$-form $\Phi\wedge\T^c$ is $\nabla^c$-parallel and the sum $\sum_i\left(e_i\hook\T\right)
\wedge\left(e_i\hook\left(\Phi\wedge\T\right)\right)$ vanishes for arbitrary $3$-forms $\T\in\Lambda^3(\R^8)$.
\end{proof}
This Lemma and (\ref{eqn:4}) result in the following proposition:
\begin{prop}
Let $\left(M^8,g,\Phi\right)$ be a $\Spin\left(7\right)$-manifold with $\nabla^c\T^c=0$. Then the Riemannian
scalar curvature $\Scal^g$ and the scalar curvature $\Scal^c$ of $\nabla^c$ are given in terms of the torsion
form by
\begin{align*}
\Scal^g&=\frac{27}{2}\,\|\T^c_8\|^2-\frac{1}{2}\,\|\T^c_{48}\|^2, & 
\Scal^c&=12\,\|\T^c_8\|^2-2\,\|\T^c_{48}\|^2.
\end{align*}
\end{prop}\noindent
A direct computation shows that for arbitrary $3$-forms $\T$, vector fields $X$ and spinor fields $\Psi$ the following
equation is satisfied:
\be
-4\,\left(X\hook\sigma^{\T}\right)\cdot\Psi=\left(\T^2-7\,\|\T_8\|^2\right)\cdot X\cdot\Psi.
\ee
The previous proposition together with this equation and the facts of \autoref{sec:2} eventually prove the following:
\begin{prop}\label{prop:1}
Let $\left(M^8,g,\Phi\right)$ be a $\Spin\left(7\right)$-manifold with $\nabla^c\T^c=0$. Any $\nabla^c$-parallel
spinor field $\Psi$ on $M^8$ satisfies
\begin{align*}
\left(\T^c\right)^2\cdot\Psi &=7\,\|\T^c_8\|^2\cdot\Psi,&
-4\,\Ric^c\left(X\right)\cdot\Psi &=\left(\left(\T^c\right)^2-7\,\|\T^c_8\|^2\right)\cdot X\cdot\Psi.
\end{align*}
\end{prop}
From now on we assume $\Spin\left(7\right)$-structures to be non-parallel, $\T^c\neq0$, and to have parallel
characteristic torsion, $\nabla^c\T^c=0$.
%
%
%
%
%
\section{Subalgebrae of \texorpdfstring{$\spin\left(7\right)$}{spin(7)}}\label{sec:4}\noindent
It is known that the group $\Spin\left(7\right)\subset\SO\left(8\right)$ acts on spinors. Let $\mrm{Cliff}\left(
\R^8\right)$ denote the real Clifford algebra of the Euclidean space $\R^8$. We will use the following real
representation of this algebra on the space of real spinors $\Delta_8:=\R^{16}$:
\be
e_i=\left[\barr{cc} 0 & M_i \\
                  M_i & 0 \earr \right] \,\,\mrm{for}\,\, i=1,\ldots,7 \,\,,\quad
e_8=\left[\barr{cc} 0 & \Id \\
                 -\Id & 0 \earr \right],
\ee
\begin{align*}
M_1&:=E_{18}+E_{27}-E_{36}-E_{45}, & M_2&:=-E_{17}+E_{28}+E_{35}-E_{46},\\
M_3&:=-E_{16}+E_{25}-E_{38}+E_{47},& M_4&:=-E_{15}-E_{26}-E_{37}-E_{48},\\
M_5&:=-E_{13}-E_{24}+E_{57}+E_{68},& M_6&:=E_{14}-E_{23}-E_{58}+E_{67},\\
M_7&:=E_{12}-E_{34}-E_{56}+E_{78}.
\end{align*}
Here $E_{i\!j}$ denotes the standard basis of the Lie algebra $\so\left(8\right)$. We fix an
orthonormal basis $\Psi_1:=\left[1,0,\ldots,0\right]^T$ , $\ldots$ , $\Psi_{16}:=\left[0,\ldots,0,1\right]^T$
of real spinors. The $4$-form $\Phi$ corresponds via the Clifford product to the real
spinor $\Psi_0:=\Psi_9-\Psi_{10}\in\Delta_8$,
\be
\Phi\cdot\Psi_0=-14\cdot\Psi_0,
\ee
and therefore $\Spin\left(7\right)$ can be seen as the isotropy group of $\Psi_0$. Its Lie algebra
$\spin\left(7\right)$ is the subalgebra of $\spin\left(8\right)$ containing all $2$-forms
\be
\omega=\sum_{i<j}\omega_{i\!j}\cdot e_{i\!j}\in\Lambda^2\left(\R^8\right)
\ee
such that the Clifford product $\omega\cdot\Psi_0=0$. This is satisfied if and only if
\begin{align*}
\omega_{18} &= -\omega_{27} + \omega_{36} + \omega_{45}, &
\omega_{28} &=  \omega_{17} + \omega_{35} - \omega_{46}, & 
\omega_{38} &= -\omega_{16} - \omega_{25} - \omega_{47}, \\
\omega_{48} &= -\omega_{15} + \omega_{26} + \omega_{37}, & 
\omega_{58} &=  \omega_{14} + \omega_{23} - \omega_{67}, &
\omega_{68} &=  \omega_{13} - \omega_{24} + \omega_{57}, \\ 
\omega_{78} &= -\omega_{12} - \omega_{34} - \omega_{56}.
\end{align*}
We fix the following basis of $\spin\left(7\right)$:
\begin{align*}
P_1&:=e_{35}+e_{46}, & P_2&:=e_{36}-e_{45}, & P_3&:=e_{15}+e_{26}, & P_4&:=e_{16}-e_{25},\\
P_5&:=e_{13}+e_{24}, & P_6&:=e_{14}-e_{23}, & P_7&:=e_{12}-e_{34}, & P_8&:=e_{34}-e_{56},
\end{align*}
\begin{align*}
Q_1&:=2\cdot e_{17}-e_{35}+e_{46}, & Q_2&:=2\cdot e_{27}+e_{36}+e_{45}, &
Q_3&:=2\cdot e_{37}+e_{15}-e_{26},\\
Q_4&:=2\cdot e_{47}-e_{16}-e_{25}, & Q_5&:=2\cdot e_{57}-e_{13}+e_{24}, &
Q_6&:=2\cdot e_{67}+e_{14}+e_{23},
\end{align*}
\begin{align*}
S_1&:=e_{18}-e_{27}, & S_2&:=e_{28}+e_{17}, & S_3&:=e_{38}-e_{47}, & S_4&:=e_{48}+e_{37},\\
S_5&:=e_{58}-e_{67}, & S_6&:=e_{68}+e_{57}, & S_7&:=e_{78}-e_{56}.
\end{align*}
For a given Lie subalgebra $\g$ of $\spin\left(7\right)$, i.e.\ $\g\leqslant\spin\left(7\right)$, we denote
by $\left(\Lambda^3\left(\R^8\right)\right)_\g$ and $\left(\Delta_8\right)_\g$ the spaces of $\g$-invariant
$3$-forms and spinors respectively. We assume $\left(\Lambda^3\left(\R^8\right)\right)_\g$ is non-trivial.
The action of $\spin\left(7\right)$ coincides on the $8$-dimensional vector spaces
\begin{align*}
\R^8&=\mrm{span}\left(e_1,\ldots,e_8\right), &
\Lambda^3_8\left(\R^8\right)&=\mrm{span}\left(\ast\left(e_1\wedge\Phi\right),\ldots,\ast\left(e_8\wedge\Phi\right)
\right).
\end{align*}
Consequently $\g$ preserves a $\T\in\left(\Lambda^3\left(\R^8\right)\right)_\g$ with $\T_8\neq0$, if and only
if it preserves a vector. A long but elementary computation for the other case $\left(\Lambda^3_8\left(\R^8\right)
\right)_\g
=\left\{0\right\}$ proves that any non-Abelian $\g$ is conjugate to
\be
\R\op\su\left(2\right)=\mrm{span}\left(P_7+2\, P_8-4\, S_7,P_5,P_6,P_7\right)<\un\left(3\right)<\su\left(4\right)<
\spin\left(7\right).
\ee
To conclude, a non-Abelian subalgebra of $\spin\left(7\right)$ that preserves a non-trivial $3$-form is either
a subalgebra of $\g_2$ or the algebra $\R\op\su\left(2\right)$ above. Dynkin's results \cites{Dyn57a,Dyn57b} on
maximal subalgebrae of exceptional Lie algebrae like $\g_2$ allow to state the following:
\begin{thm}\label{thm:1}
Let $\g$ be a non-Abelian subalgebra of $\spin\left(7\right)$. If there exists a non-trivial $\g$-invariant $3$-form
$\T$, i.e.\ $0\neq \T\in\left(\Lambda^3\left(\R^8\right)\right)_\g$, then $\g$ is conjugate to one of the following
algebrae:
\begin{align*}
&\g_2 & &= \!\!& &\mrm{span}\left(P_1,\ldots,P_8,Q_1,\ldots,Q_6\right) & &< \!\!& &\spin\left(7\right),\\
&\su\left(3\right) & &= \!\!& &\mrm{span}\left(P_1,\ldots,P_8\right) & &< \!\!& &\g_2,\\
&\su\left(2\right)\op\su_c\left(2\right)\!\!\! & &= \!\!& &\mrm{span}\left(P_5,P_6,P_7,P_7+2\, P_8,Q_5,Q_6\right) & &< \!\!& &\g_2,\\
&\un\left(2\right) & &= \!\!& &\mrm{span}\left(P_7+2\, P_8,P_5,P_6,P_7\right) & &< \!\!& &\su\left(3\right),\\
&\R\op\su_c\left(2\right) & &= \!\!& &\mrm{span}\left(P_7,P_7+2\, P_8,Q_5,Q_6\right) & &< \!\!& &\su\left(2\right)\op\su_c\left(2\right),\\
&\so\left(3\right) & &= \!\!& &\mrm{span}\left(P_1+P_5,P_2+P_6,P_7+P_8\right) & &< \!\!& &\su\left(3\right),\\
&\su\left(2\right) & &= \!\!& &\mrm{span}\left(P_5,P_6,P_7\right) & &< \!\!& &\un\left(2\right),\\
&\su_c\left(2\right) & &= \!\!& &\mrm{span}\left(P_7+2\, P_8,Q_5,Q_6\right) & &< \!\!& &\R\op\su_c\left(2\right),\\
&\so_{ir}\left(3\right) & &= \!\!& &\mrm{span}\left(P_5-{\sst\sqrt{3/5}}\, Q_2,P_6+
{\sst\sqrt{3/5}}\, Q_1,P_7+3\, P_8\right)\!\!\! & &< \!\!& &\g_2,\\
&\R\op\su\left(2\right) & &= \!\!& &\mrm{span}\left(P_7+2\, P_8-4\, S_7,P_5,P_6,P_7\right) &
&< \!\!& &\su\left(4\right)<\spin\left(7\right).
\end{align*}
\end{thm}\noindent
Here $\su_c\left(2\right)$ denotes the centralizer of $\su\left(2\right)$ inside $\g_2$ which is isomorphic, but not
conjugate, to $\su\left(2\right)$. $\so_{ir}\left(3\right)$ denotes the maximal subalgebra of $\g_2$ generating an
irreducible $7$-dimensional real representation.

The Lie algebrae $\g_2$ and $\R\op\su\left(2\right)$ are of rank $2$. Their maximal tori are given by
\begin{align*}
\mfr{t}^2 &:= k\cdot P_7 + l\cdot \left(P_7+2\, P_8\right)<\g_2, &
\tilde{\mfr{t}}^2 &:= \tilde{k}\cdot P_7 + \tilde{l}\cdot \left(P_7+2\,P_8-4\,S_7\right)<\R\op\su\left(2\right).
\end{align*}
$1$-dimensional tori contained in these will be denoted by $\mfr{t}^1$ or $\tilde{\mfr{t}}^1$
respectively.
%
%
%
%
\section{Algebraic classification}\label{sec:5}\noindent
Given a non-parallel $\Spin\left(7\right)$-structure let $\iso\left(\T^c\right)$ be the isotropy algebra of the
characteristic torsion $\T^c$ and $\hol\left(\nabla^c\right)$ the holonomy algebra of the characteristic
connection $\nabla^c$. Obviously, these two are Lie subalgebrae of $\spin\left(7\right)$, and a non-Abelian
$\iso\left(\T^c\right)$ is one of the algebrae in theorem \ref{thm:1}. But not all of those algebrae can
occur as the isotropy algebra of a non-trivial $3$-form. A direct computation proves the following:
\begin{prop}
If the isotropy algebra $\iso\left(\T\right)<\spin\left(7\right)$ of a non-trivial $3$-form $\T$ contains
$\su_c\left(2\right)$ or $\su\left(2\right)$, then $\dim\left(\iso\left(\T\right)\right)\geq 4$.
\end{prop}\noindent
Since we restricted the consideration to parallel characteristic torsion the holonomy algebra
$\hol\left(\nabla^c\right)$ is a subalgebra of $\iso\left(\T^c\right)$,
\be
\hol\left(\nabla^c\right)\leqslant\iso\left(\T^c\right)<\spin\left(7\right).
\ee

Conversely, fix $\h\leqslant\g<\spin\left(7\right)$. Suppose there exists a $\Spin\left(7\right)$-manifold
with $\hol\left(\nabla^c\right)=\h$ and $\nabla^c$-parallel torsion $\T^c\neq0$ satisfying $\iso\left(\T^c\right)=\g$. 
Then $\T^c$ is necessarily contained in the space of $\g$-invariant $3$-forms $\left(\Lambda^3\left(\R^8\right)
\right)_\g$ satisfying
\begin{align}\label{eqn:2}
\left(\T^c\right)^2\cdot\Psi &=7\,\|\T^c_8\|^2\cdot\Psi,&
-4\,\Ric^c\left(X\right)\cdot\Psi &=\left(\left(\T^c\right)^2-7\,\|\T^c_8\|^2\right)\cdot X\cdot\Psi \tag{$\bullet$}
\end{align}
for all $\h$-invariant spinors $\Psi\in\left(\Delta_8\right)_\h$ and all vectors $X\in\R^8$ (cf.\ proposition
\ref{prop:1}). Furthermore, two torsion forms $0\neq\T^c_1,\T^c_2\in\left(\Lambda^3\left(\R^8\right)\right)_\g$ define
equivalent geometric structures if they are equivalent under the action of the algebra
\be
\mfr{inv}\left(\Lambda^3\left(\R^8\right)\right)_\g:=\left\{x\in\spin\left(7\right)\,\,:\,\,x\left(
\Lambda^3\left(\R^8\right)\right)_\g\subseteq\left(\Lambda^3\left(\R^8\right)\right)_\g\right\}.
\ee
Define the \emph{space $\mca{K}\left(\h\right)$ of algebraic curvature tensors with values in $\h$} by
\be
\mca{K}\left(\h\right):=\left\{\mrm{R}\in\Lambda^2\left(\R^8\right)\ox\h\,\,:\,\,
\sigma_{X,Y,Z}\left\{\mrm{R}\left(X,Y,Z,V\right)\right\}=0\,\,\,\all\,X,Y,Z,V\in\R^8\right\}.
\ee
Here $\sigma_{X,Y,Z}$ denotes the cyclic sum over $X,Y,Z$. If the space $\mca{K}\left(\hol\left(\nabla^c
\right)\right)$ is trivial for a $\Spin\left(7\right)$-manifold with parallel torsion, the curvature operator $\mrm{R}^c\,:\,\Lambda^2\left(\R^8\right)\ra\hol\left(\nabla^c\right)$ of the characteristic connection is
$\nabla^c$-parallel (see \cite{CS04}) and thus $\hol\left(\nabla^c\right)$-invariant. A case-by-case study
proves the following:
\begin{prop}\label{prop:2}
Let $\h\leqslant\g<\spin\left(7\right)$ with $\g$ non-Abelian and suppose there exists a non-trivial
$\g$-invariant $3$-form. Then $\mca{K}\left(\h\right)$ is non-trivial if and only if $\su\left(2\right)\leqslant\h$.
\end{prop}
The recipe to obtain necessary conditions on $\T^c$ and $\Ric^c$ goes as follows:
\bnum
\item Fix $\h=\hol\left(\nabla^c\right)\leqslant\iso\left(\T^c\right)=\g$ with $\g<\spin\left(7\right)$
      non-Abelian.
\item Determine the spaces $\left(\Lambda^3\left(\R^8\right)\right)_\g$ and $\left(\Delta_8\right)_\h$.
\item Solve (\ref{eqn:2}).
\item Quotient out the action of $\mfr{inv}\left(\Lambda^3\left(\R^8\right)\right)_\g$ on $\T^c\neq0$.
\item If $\su\left(2\right)\nleqslant\h$, analyze the $\h$-invariance and the symmetry of $\mrm{R}^c$.
\enum
Applying this, we determine $\T^c$ and $\Ric^c$ for all admissible combinations of $\hol\left(\nabla^c\right)$ and
non-Abelian $\iso\left(\T^c\right)$,
\be
\iso\left(\T^c\right) = \g_2,\,\, \su\left(3\right),\,\, \su\left(2\right)\op\su_c\left(2\right),\,\, \un\left(2\right),
\,\, \R\op\su_c\left(2\right),\,\,\so\left(3\right),\,\, \so_{ir}\left(3\right),\,\, \R\op\su\left(2\right).
\ee
The condition of (5), $\su\left(2\right)\nleqslant\hol\left(\nabla^c\right)$, is satisfied for $\hol\left(\nabla^c\right)
\leqslant\R\op\su_c\left(2\right)$, $\so\left(3\right)$, $\so_{ir}\left(3\right)$ or $\tilde{\mfr{t}}^2$. For clarity
we define the following forms:
\begin{align*}
Z_1&:=e_{12}+e_{34}, & Z_2&:=e_{56}, & Z_3&:=e_{12}-e_{34}, & D_1&:=e_{246}-e_{145},\\
D_2&:= -e_{235}-e_{136}, & D_3&:= -e_{135}+e_{245}, & D_4&:= e_{146}+e_{236}, & D_5&:= e_{123}-e_{356},
\end{align*}
so that we have
\begin{align*}
Z &= Z_1 + Z_2, & D &= D_1 + D_2, & \bar{D} &:= D_3 + D_4.
\end{align*}
\subsection{The cases \texorpdfstring{$\bm{\iso\left(\mathbf{T}^c\right)=\g_2}$, $\bm{\su\left(2\right)\op\su_c\left(2\right)}$, 
$\bm{\R\op\su_c\left(2\right)}$, $\bm{\so_{ir}\left(3\right)}$}{iso(T^c)=g_2, su(2)+su_c(2), R+su_c(2), so_ir(3)}}
The characteristic torsion is an element of the family
\be
\T^c=a_1\cdot\left( Z \wedge e_7 + D \right) + b_1\cdot\left( \left(Z_1-6\,Z_2\right)\wedge e_7 + D \right)
+ b_2\cdot\left(Z_3\wedge e_8\right),
\ee
where $a_1,b_1,b_2\in\R$. The constraints on the torsion parameters relative to the considered
isotropy algebrae are arranged in the following table:
\btab{|c|c|c|c|c}
\cline{1-4}
$\iso\left(\T^c\right)$ & $\g_2$, $\so_{ir}\left(3\right)$ & $\su\left(2\right)\op\su_c\left(2\right)$ &
$\R\op\su_c\left(2\right)$&\\
\cline{1-4}
constraints & $b_1=b_2=0$ & $b_1\neq0$, $b_2=0$ & $b_2\neq0$&$\!\!.$\\
\cline{1-4}
\etab

The corresponding $\Spin\left(7\right)$-structure is of type $\W_1$ or $\W_2$ if and only if $a_1=0$ or
$b_1=b_2=0$ respectively. The characteristic Ricci tensor $\Ric^c$ has the shape
\be
\Ric^c = \diag\left(\lambda,\lambda,\lambda,\lambda,\kappa,\kappa,\kappa,0\right)
\ee
depending on the parameters of the torsion form,
\begin{align*}
\lambda &= 3\,\left(a_1+b_1\right)\left(4\,a_1-3\,b_1\right) - b_2^2, & \kappa &= 4\,\left(a_1+b_1\right)\left(3\,a_1-4\,b_1\right).
\end{align*}

We proceed with the holonomy classification. System (\ref{eqn:2}) becomes inconsistent for
$\su\left(2\right)\leqslant\h\leqslant\g=\g_2$. Moreover, we deduce $\kappa=0$ in the case of $\hol\left(
\nabla^c\right)=\un\left(2\right)$ or $\su\left(2\right)$ with $\iso\left(\T^c\right)=\su\left(2\right)\op
\su_c\left(2\right)$.
\subsubsection{The subcases \texorpdfstring{$\hol\left(\nabla^c\right)\leqslant\R\op\su_c\left(2\right)$, $\so\left(3\right)$}
{hol(\nabla^c)<=R+su_c(2), so(3)}}
Applying step (5) we are able to compute the characteristic curvature tensor
\be
\mrm{R}^c = r_1\cdot\left( P_7\ox P_7 \right) 
+ r_2\cdot\left( \left(P_7+2\,P_8\right)\ox\left(P_7+2\,P_8\right) + Q_5\ox Q_5 + Q_6 \ox Q_6 \right),
\ee
which depends on the torsion parameters in the following way:
\begin{align*}
&r_1 =  \frac{3}{8}\,\kappa-\lambda = -\frac{3}{2}\,\left(a_1+b_1\right)\left(5\,a_1-2\,b_1\right) + b_2^2,
&r_2 &= -\frac{1}{8}\,\kappa.
\end{align*}
We arranged the necessary conditions on the parameters $r_1$ and $r_2$ for each holonomy
algebra $\hol\left(\nabla^c\right)\leqslant\R\op\su_c\left(2\right)$ or $\so\left(3\right)$ in the following table:
\btab{|c|c|c|c|c}
\cline{1-4}
$\hol\left(\nabla^c\right)$ & $\su_c\left(2\right)$ & $\mfr{t}^2$, $\mfr{t}^1\left[l=0\right]$ & $\so\left(3\right)$,
$\mfr{t}^1\left[l\neq0\right]$, $\mfr{0}$&\\
\cline{1-4}
constraints & $r_1=0$ & $r_2=0$ & $r_1=r_2=0$&$\!\!.$\\
\cline{1-4}
\etab
Here $\mfr{0}$ denotes the zero algebra.
\subsubsection{The subcase \texorpdfstring{$\hol\left(\nabla^c\right)=\so_{ir}\left(3\right)$}{hol(\nabla^c)=so_ir(3)}}
There exists only one $\so_{ir}\left(3\right)$-invariant curvature tensor $\mrm{R}^c\,:\,\Lambda^2\left(\R^8\right)
\ra\so_{ir}\left(3\right)$, namely the projection onto the algebra $\so_{ir}\left(3\right)$,
\begin{align*}
\mrm{R}^c = -a_1^2\cdot\left( U_1\ox U_1 + U_2\ox U_2 + U_3\ox U_3 \right).
\end{align*}
Here $\left(U_1,U_2,U_3\right)$ denotes the following basis of $\so_{ir}\left(3\right)$:
\begin{align*}
U_1&:= \sqrt{5/2}\,P_5-\sqrt{3/2}\, Q_2, & U_2&:= \sqrt{5/2}\,P_6+\sqrt{3/2}\, Q_1, &
U_3&:= P_7+3\,P_8.
\end{align*}
\subsection{The cases \texorpdfstring{$\bm{\iso\left(\mathbf{T}^c\right)=\so\left(3\right)}$, $\bm{\su\left(3
\right)}$}{iso(T^c)=so(3), su(3)}}
There are two admissible families of characteristic torsions. The first one depends on a single
positive parameter,
\be
\T^c_{I} = a_1\cdot Z\wedge e_7, \quad a_1\in\R, \,\,a_1>0,
\ee
whilst the second is a $3$-parameter family
\be
\T^c_{I\!I} = a_1\cdot\bar{D} + a_2\cdot\left(2\,D_1+5\,D_2+3\,D_5\right)
           + b_1\cdot\left(D_1-\,D_2-2\,D_5\right)
\ee
with $a_1,a_2,b_1\in\R$, $b_1>0$. The isotropy algebra of type $I$ is the algebra $\su\left(3\right)$,
i.e.\ $\iso\left(\T^c_I\right)=\su\left(3\right)$. The condition $\iso\left(\T^c_{I\!I}\right)=\su\left(3\right)$
holds if and only if $a_1=0$, $b_1=\frac{3}{2}\,a_2$,
\be
\T^c_{I\!I}\sim D=D_1+D_2.
\ee

The $\Spin\left(7\right)$-structures of this subsection are not of type $\W_2$. They are of type $\W_1$ if only
if $\T^c$ is of type $I\!I$ and $a_1=a_2=0$. The characteristic Ricci tensor is
\be
\Ric^c = \diag\left(\lambda,\lambda,\lambda,\lambda,\lambda,\lambda,0,0\right)
\ee
and depends on the torsion type
\begin{align*}
\lambda_{I} &= 2\,a_1^2, &
\lambda_{I\!I} &= 4\,a_1^2+4\,\left(2\,a_2+b_1\right)\left(5\,a_2-b_1\right)
\end{align*}
accordingly.
\subsubsection{The subcase \texorpdfstring{$\hol\left(\nabla^c\right)\leqslant\so\left(3\right)$}{hol(\nabla^c)<=so(3)}}
The curvature operator $\mrm{R}^c\,:\,\Lambda^2\left(\R^8\right)\ra\so\left(3\right)$ is the projection onto the
subalgebra $\so\left(3\right)<\spin\left(7\right)$ scaled by the parameter $\lambda$ above,
\be
\mrm{R}^c = -\frac{\lambda}{2}\cdot\left( V_1\ox V_1 + V_2\ox V_2 + V_3\ox V_3 \right).
\ee
The basis $\left(V_1,V_2,V_3\right)$ of $\so\left(3\right)$ is
\begin{align*}
V_1&:= \sqrt{1/2}\,\left(P_1+P_5\right), & V_2&:= \sqrt{1/2}\,\left(P_2+P_6\right), &
V_3&:= P_7+P_8.
\end{align*}
For $\hol\left(\nabla^c\right)=\mfr{t}^1<\so\left(3\right)$ and in the case of trivial holonomy, $\hol\left(\nabla^c
\right)=\mfr{0}$, the parameter $\lambda$ has to vanish necessarily.
\subsubsection{The subcase \texorpdfstring{$\hol\left(\nabla^c\right)=\mfr{t}^2$}{hol(\nabla^c)=t^2}}
The characteristic curvature $\mrm{R}^c\,:\,\Lambda^2\left(\R^8\right)\ra\mfr{t}^2$ is
\be
\mrm{R}^c = -\frac{\lambda}{4}\cdot\left( 3\left( P_7\ox P_7 \right) + \left(P_7+2\,P_8\right)\ox\left(P_7+2\,P_8\right) \right)
\ee
for the parameter $\lambda$ above.
\subsection{The case \texorpdfstring{$\bm{\iso\left(\mathbf{T}^c\right)=\un\left(2\right)}$}{iso(T^c)=u(2)}}
The characteristic torsion is an element of one of the following two $3$-parameter families:
\be
\T^c_{I} = a_1\cdot\left(Z_1+5\,Z_2\right)\wedge e_8 + a_2\cdot\left(Z_1+5\,Z_2\right)\wedge e_7 
        + b_1\cdot\left(Z_1-2\,Z_2\right)\wedge e_7,
\ee
\be
\T^c_{I\!I} = a_1\cdot\left( \left(Z_1-2\,Z_2\right)\wedge e_8 + \frac{7}{4}\,\bar{D} \right) 
           + a_2\cdot\left( \left(Z_1-2\,Z_2\right)\wedge e_7 + \frac{7}{4}\,D \right)
           + b_1\cdot\left(Z_1-2\,Z_2\right)\wedge e_7.
\ee
Here $a_1,a_2,b_1\in\R$, $b_1>0$. Not all parameter configurations are admissible: the isotropy algebra
$\iso\left(\T^c_I\right)$ of type $I$ contains the algebra $\su\left(2\right)\op\su_c\left(2\right)$
if $a_1=0$ and $b_1=-a_2$, and the condition $\su\left(2\right)\op\su_c\left(2\right)\leqslant\iso\left(\T^c_{I\!I}
\right)$ holds if $a_1=0$ and $b_1=\frac{3}{4}\,a_2$. The isotropy algebra is $\iso\left(\T^c\right)=\su\left(3\right)$
if and only if $a_1=0$ and $b_1=\frac{4}{3}\,a_2$ (for type $I$) or $a_1=0$ and $b_1=-a_2$ (for type
$I\!I$) respectively. These four cases have to be excluded. 

The considered $\Spin\left(7\right)$-structures are not of type $\W_2$. Those of type $\W_1$ satisfy $a_1=a_2=0$.
In this particular case both torsion families coincide, i.e.\ $\T^c_{I}=\T^c_{I\!I}$. The Ricci tensor
of $\nabla^c$ is
\be
\Ric^c = \diag\left(\lambda,\lambda,\lambda,\lambda,\kappa,\kappa,0,0\right).
\ee
The constants $\lambda$ and $\kappa$ depend on the torsion type,
\begin{align*}
\lambda_{I} &= 6\,a_1^2+\left(a_2+b_1\right)\left(6\,a_2-b_1\right), &
\kappa_{I} &= 10\,a_1^2+2\,\left(a_2+b_1\right)\left(5\,a_2-2\,b_1\right),\\
\lambda_{I\!I} &= \frac{45}{4}\,\left(a_1^2+a_2^2\right)-2\,a_2\,b_1-b_1^2, &
\kappa_{I\!I} &= \frac{33}{4}\,\left(a_1^2+a_2^2\right)-8\,a_2\,b_1-4\,b_1^2 .
\end{align*}
\subsubsection{The subcase \texorpdfstring{$\hol\left(\nabla^c\right)\leqslant\mfr{t}^2$}{hol(\nabla^c)<=t^2}}
Proposition \ref{prop:2} allows to compute the curvature tensor of the characteristic connection,
\be
\mrm{R}^c = r_1\cdot\left( P_7\ox P_7 \right) + r_2\cdot\left(P_7+2\,P_8\right)\ox\left(P_7+2\,P_8\right),
\ee
where $r_1$ and $r_2$ are given in terms of the parameters $\lambda$ and $\kappa$ above,
\begin{align*}
r_1 &= \frac{1}{4}\,\kappa-\lambda, & r_2 &= -\frac{1}{4}\,\kappa.
\end{align*}
Conditions on these parameters, given a specific holonomy algebra $\hol\left(\nabla^c\right)\leqslant\mfr{t}^2$,
are the following:
\btab{|c|c|c|c|c}
\cline{1-4}
$\hol\left(\nabla^c\right)$ & $\mfr{t}^1\left[k=0\right]$ & $\mfr{t}^1\left[l=0\right]$  & $\mfr{t}^1\left[k,l\neq0\right]$,
$\mfr{0}$&\\
\cline{1-4}
constraints & $r_1=0$ & $r_2=0$ & $r_1=r_2=0$&$\!\!.$\\
\cline{1-4}
\etab
The condition $r_1=r_2=0$ can only be realized for $\T^c_I$ with $a_1=0$ and $b_1=-a_2$, one of
the excluded possibilities. Consequently, there exists no $\Spin\left(7\right)$-structure with
parallel characteristic torsion, $\iso\left(\T^c\right)=\un\left(2\right)$ and $\hol\left(
\nabla^c\right)=\mfr{t}^1\left[k,l\neq0\right]$ or $\mfr{0}$.
\subsection{The case \texorpdfstring{$\bm{\iso\left(\mathbf{T}^c\right)=\R\op\su\left(2\right)}$}{iso(T^c)=R+su(2)}}
Here the characteristic torsion form $\T^c$ is an element of the $1$-parameter family
\be
\T^c = b_1\cdot\left(D_3-D_4\right), \quad b_1\in\R, \,\,b_1>0,
\ee
and $\Ric^c$ is given by
\be
\Ric^c = \diag\left(0,0,0,0,-4\,b_1^2,-4\,b_1^2,0,0\right).
\ee
The corresponding $\Spin\left(7\right)$-structure is of type $\W_1$. The Ricci tensor of a $\tilde{\mfr{t}}^2$-invariant
and symmetric curvature operator is an element of the $3$-parameter family
\be
\diag\left(\alpha+\beta+\gamma,\alpha+\beta+\gamma,\alpha+\beta-\gamma,\alpha+\beta-\gamma,4\,\beta,4\,\beta,
16\,\beta,16\,\beta\right),\quad \alpha,\beta,\gamma\in\R.
\ee
Thus $\R\op\su\left(2\right)$ is the only admissible (i.e.\ $\T^c\neq0$) characteristic holonomy algebra for
$\iso\left(\T^c\right)=\R\op\su\left(2\right)$.
\subsection{The admissible isotropy and holonomy algebrae}\label{ssec:5.5}
Summarizing the previous subsections, the following table provides an overview of the isotropy and
holonomy algebrae which comply with the requirements of steps (1) to (5) and lead to non-vanishing
characteristic torsion:
\btab{|c||c|c|c}
\cline{1-3}
$\iso\left(\T^c\right)$ & \multicolumn{2}{c|}{$\hol\left(\nabla^c\right)$}&\\
\cline{2-3}
& $\mca{K}\left(\hol\left(\nabla^c\right)\right)\neq0$ & $\mca{K}\left(\hol\left(\nabla^c\right)\right)=0$
($\Ra\!\nabla^c\mrm{R}^c=0$)&\\
\cline{1-3}\cline{1-3}
$\g_2$ & $\g_2$, $\su\left(2\right)\op\su_c\left(2\right)$ & $\R\op\su_c\left(2\right)$, $\so_{ir}\left(3\right)$&\\
\cline{1-3}
$\su\left(3\right)$ & $\su\left(3\right)$, $\un\left(2\right)$ & $\so\left(3\right)$, $\mfr{t}^2$&\\
\cline{1-3}
$\su\left(2\right)\op\su_c\left(2\right)$ & $\su\left(2\right)\op\su_c\left(2\right)$, $\un\left(2\right)$, $\su\left(
2\right)$ & $\R\op\su_c\left(2\right)$, $\su_c\left(2\right)$, $\so\left(3\right)$, $\mfr{t}^2$, $\mfr{t}^1$, $\mfr{0}$&\\
\cline{1-3}
$\un\left(2\right)$ & $\un\left(2\right)$, $\su\left(2\right)$ & $\mfr{t}^2$, $\mfr{t}^1$&\\
\cline{1-3}
$\R\op\su_c\left(2\right)$ & --- & $\R\op\su_c\left(2\right)$, $\su_c\left(2\right)$, $\mfr{t}^2$, $\mfr{t}^1$, $\mfr{0}$&\\
\cline{1-3}
$\so\left(3\right)$ & --- & $\so\left(3\right)$, $\mfr{t}^1$, $\mfr{0}$&\\
\cline{1-3}
$\so_{ir}\left(3\right)$ & --- & $\so_{ir}\left(3\right)$&\\
\cline{1-3}
$\R\op\su\left(2\right)$ & $\R\op\su\left(2\right)$ & ---&$\!\!.$\\
\cline{1-3}
\etab
%
%
%
%
\section{Geometric results}\label{sec:6}\noindent
In this section we discuss the geometries related to the algebraic cases of \autoref{sec:5}. The most
important tool in these considerations is the splitting theorem of de Rham generalized to geometric
structures with totally skew-symmetric torsion (see \cite{CM08}):
\begin{thm}
Let $\left(M^n,g,\T\right)$ be a complete, simply connected Riemannian manifold with $3$-form $\T$. Suppose
the tangent bundle
\be
TM^n=TM_+\op TM_-
\ee
splits under the action of the holonomy group of $\nabla_X Y = \nabla^g_X Y + \frac{1}{2}\cdot \T\left(X,Y,\,
\cdot\,\right)$ so that
\begin{align}\label{eqn:3}
\T\left(X_+,X_-,\,\cdot\,\right)&=0, & \T\left(X_+,Y_+,\,\cdot\,\right)&\in TM_+,
& \T\left(X_-,Y_-,\,\cdot\,\right)&\in TM_-\tag{$\ast$}
\end{align}
for all $X_+,Y_+\in TM_+$ and $X_-,Y_-\in TM_-$. Let $\T=\T_++\T_-$ denote the corresponding
decomposition of the $3$-form $\T$. Then $\left(M,g,\T\right)$ is isometric to a Riemannian product
\be
\left(M_+,g_+,\T_+\right)\x\left(M_-,g_-,\T_-\right).
\ee
The condition $\nabla\T=0$ results in $\nabla^+\T_+=0$, $\nabla^-\T_-=0$ for
\begin{align*}
\nabla^+_X Y &:= \nabla^{g_+}_X Y + \frac{1}{2}\cdot \T_+\left(X,Y,\,\cdot\,\right), &
\nabla^-_X Y &:= \nabla^{g_-}_X Y + \frac{1}{2}\cdot \T_-\left(X,Y,\,\cdot\,\right).
\end{align*}
\end{thm}

We split the consideration in the same manner as in \autoref{sec:5} and start with the most
obvious cases.
\subsection{The case \texorpdfstring{$\bm{\iso\left(\mathbf{T}^c\right)=\so_{ir}\left(3\right)}$}{iso(T^c)=so_ir(3)}}
The characteristic holonomy $\hol\left(\nabla^c\right)$ is equal to $\so_{ir}\left(3\right)$, the $\nabla^c$-parallel
torsion form $\T^c$ is proportional to $\left(Z\wedge e_7+D\right)$ and the tangent bundle of $\left(M^8,g,\Phi\right)$
splits into the following $\so_{ir}\left(3\right)$-invariant components:
\be
TM^8=E \op \R\cdot e_8.
\ee
There exist two spinor fields which are parallel with respect to $\nabla^c$, namely $\Psi_1-\Psi_2$
and $\Psi_9-\Psi_{10}$ (cf.\ \autoref{sec:4}). The curvature tensor $\mrm{R}^c$ is uniquely determined,
$\so_{ir}\left(3\right)$-invariant, $\nabla^c$-parallel and $\mrm{R}^c=\mrm{R}^c|_E\,\op\,0|_{\R\cdot e_8}$
(see \autoref{sec:5}). Since the torsion form $\T^c$ does not depend on $e_8$ and $\nabla^ce_8=0$,
we conclude that $e_8$ is $\nabla^g$-parallel. Consequently, a complete and simply connected $M^8$
is the Riemannian product of a $7$-dimensional manifold $Y^7$ with $\R$. We furthermore conclude
that the space $Y^7$ is isometric to a naturally reductive, nearly parallel $\G_2$-structure with
fundamental form $\left(Z\wedge e_7+D\right)$ and characteristic holonomy algebra $\so_{ir}\left(3
\right)$ (see \cite{FI02}). Consider the embedding of $\SO\left(3\right)$ into $\SO\left(5\right)$
given by the $5$-dimensional irreducible $\SO\left(3\right)$-representation. This gives rise to the
homogeneous naturally reductive space $\SO\left(5\right)/\SO_{ir}\left(3\right)$. With \cites{Fri06,
FKMS97} we obtain that $Y^7$ is isometric to $\SO\left(5\right)/\SO_{ir}\left(3\right)$.
\begin{thm}\label{thm:2}
A complete, simply connected $\Spin\left(7\right)$-manifold with parallel characteristic torsion, $\nabla^c\T^c=0$,
and $\iso\left(\T^c\right)=\so_{ir}\left(3\right)$ is isometric to the Riemannian product of $\SO\left(5\right)/
\SO_{ir}\left(3\right)$ with $\R$.
\end{thm}
\subsection{The cases \texorpdfstring{$\bm{\iso\left(\mathbf{T}^c\right)=\so\left(3\right)}$, $\bm{\su\left(3
\right)}$}{iso(T^c)=so(3), su(3)}}
The $\Spin\left(7\right)$-structure $\left(M^8,g,\Phi\right)$ admits two $\nabla^c$-parallel vector fields $e_7$,
$e_8$ and four $\nabla^c$-parallel spinor fields $\Psi_1$, $\Psi_2$, $\Psi_9$, $\Psi_{10}$. Moreover,
the differential forms $Z$ and $D$ are parallel with respect to $\nabla^c$, and we can reconstruct a
$\Spin\left(7\right)$-structure by using (\ref{eqn:1}).

There are two types of characteristic torsion. First we discuss type $I$: $\T^c\sim Z\wedge e_7$ and
$\iso\left(\T^c\right)=\su\left(3\right)$. The torsion form vanishes along $e_8$, i.e.\ $e_8\hook \T^c=0$,
and the tangent bundle splits into two $\hol\left(\nabla^c\right)$-invariant components,
\be
TM^8=E\op\R\cdot e_8.
\ee
Consequently, a complete, simply connected $M^8$ is isometric to the Riemannian product of a $7$-dimensional
manifold $Y^7$ with $\R$. The torsion $3$-form $\T^c$ is contained in the space $\Lambda^3_1\left(\R^7\right)\op
\Lambda^3_{27}\left(\R^7\right)$ of the decomposition of $\Lambda^3\left(\R^7\right)$ into irreducible $\G_2
$-components. Up to isometry the space $Y^7$ admits a cocalibrated $\G_2$-structure with fundamental form $\left(
Z\wedge e_7+D\right)$ and a characteristic connection with totally skew-symmetric torsion equal to $\T^c$. The
holonomy algebra of this connection coincides with $\hol\left(\nabla^c\right)$. Consequently, $Y^7$ is homothetic
to an $\eta$-Einstein Sasakian $7$-manifold with contact vector field $e_7$ and fundamental form $Z$ (see \cite{Fri06}).
\begin{thm}
Let $\left(M^8,g,\Phi\right)$ be a complete, simply connected $\Spin\left(7\right)$-manifold with parallel characteristic
torsion $\T^c$ and $\iso\left(\T^c\right)=\su\left(3\right)$. Suppose that the torsion form is of type $I$, i.e.\ $\T^c\sim
Z\wedge e_7$. Then $M^8$ is isometric to the Riemannian product of a $7$-dimensional, simply connected, $\eta$-Einstein
\be
\Ric^{\bar{g}}=10\cdot \bar{g}-4\cdot e_7\ox e_7
\ee
Sasakian manifold $\left(Y^7,\bar{g},e_7,Z\right)$ with $\R$. Conversely, such a product admits a $\Spin\left(7\right)
$-structure with parallel characteristic torsion and $\hol\left(\nabla^c\right)$ contained in $\su\left(3\right)$.
\end{thm}
\begin{rmk}
Simply connected Sasakian manifolds $\left(M^7,g,\xi,\vphi\right)$ which admit the Ricci tensor of the last theorem can
be constructed via the Tanno deformation of a $7$-dimensional Einstein-Sasakian structure $\left(\tilde{M}^7,\tilde{
g},\tilde{\xi},\tilde{\vphi}\right)$. This deformation is
\be
\vphi:=\tilde{\vphi},\quad \xi:=a^2\cdot\tilde{\xi},
\quad g:=a^{-2}\cdot\tilde{g}+\left(a^{-4}-a^{-2}\right)\cdot\tilde{\eta}\ox\tilde{\eta}
\ee
with the deformation parameter $a^2=\frac{3}{2}$ (see \cite{FK00}). We recommend the article \cite{BGM04} for further
constructions of Sasakian structures of $\eta$-Einstein type.
\end{rmk}
\begin{exa}
The algebraic classification in \autoref{sec:5} proves that for $\hol\left(\nabla^c\right)=\so\left(3\right)$ or $\mfr{t}^2$
the corresponding $\Spin\left(7\right)$-structure $M^8$ is a homogeneous naturally reductive space. Since the curvature
tensor of the characteristic connection does not depend on $e_8$, we can conclude the same for the Sasakian manifold $Y^7$
and denote $Y^7=\G/\mrm{H}$. In \cite{Fri06} $Y^7$ was identified for characteristic holonomy $\h=\hol\left(\nabla^c\right)
=\so\left(3\right)$. Here the corresponding naturally reductive space is isometric to the Stiefel manifold $Y^7=\SO\left(5
\right)/\SO\left(3\right)$. We finally discuss the case $\h=\hol\left(\nabla^c\right)=\mfr{t}^2$. The Lie algebra $\g$ of
the $9$-dimensional automorphism group $\G$ is given by $\g=\mfr{t}^2\op\R^7$ with the bracket
\be
\left[A+X,B+Y\right]=\left(\left[A,B\right]-\mrm{R}^c\left(X,Y\right)\right)+\left(A\cdot Y-B\cdot X-\T^c\left(X,Y\right)\right).
\ee
It turns out that the corresponding Killing form is non-degenerate, and thus $\g$ is semisimple. Consequently,
$\g$ is isomorphic to $\su\left(2\right)\op c\left(\su\left(2\right)\right)$, where $c\left(\su\left(2\right)\right)$ denotes
the centralizer of $\su\left(2\right)$ inside $\spin\left(7\right)$.
\end{exa}
We proceed with torsion type $I\!I$. The torsion form does not depend on $e_7$ and $e_8$, and the tangent
bundle splits into the following $\hol\left(\nabla^c\right)$-invariant components:
\be
TM^8=E_1\op E_2.
\ee
Here $E_2$ is spanned by $\left\{e_7,e_8\right\}$. The torsion form $\T^c$ belongs to the $\Lambda^3_2\left(\R^6\right)\op
\Lambda^3_{12}\left(\R^6\right)$-component of the decomposition of $\Lambda^3\left(\R^6\right)$ under the action of $\Un
\left(3\right)$ (see \cite{AFS05}). Consequently, if $M^8$ is simply connected and complete, then it is isometric to the
Riemannian product of $\R^2$ with an almost Hermitian manifold $X^6$ of Gray-Hervella type $W_1\op W_3$, with K\"ahler form
$Z$ and characteristic holonomy contained in $\iso\left(\T^c\right)$. The latter structures have been exhaustively studied
in \cites{AFS05,Sch07}.
\begin{thm}
A complete, simply connected $\Spin\left(7\right)$-manifold with parallel characteristic torsion of type $I\!I$ in the
class $\iso\left(\T^c\right)=\so\left(3\right)$ or $\su\left(3\right)$ is isometric to the Riemannian product $X^6\x\R^2$,
where $X^6$ is an almost Hermitian manifold of Gray-Hervella type $W_1\op W_3$ with characteristic holonomy contained in
$\iso\left(\T^c\right)$.
\end{thm}
\begin{rmk}
Consider the special case $\hol\left(\nabla^c\right)=\su\left(3\right)=\iso\left(\T^c\right)$ and torsion type $I\!I$. Here
$\T^c$ is proportional to $D$ and $X^6$ is isometric to a strictly (i.e.\ non-K\"ahler) nearly K\"ahler manifold. Conversely,
any Riemannian product $X^6\x\R^2$ with $X^6$ a strictly nearly K\"ahler manifold admits a $\Spin\left(7\right)$-structure
with parallel characteristic torsion and characteristic holonomy contained in $\su\left(3\right)$.
\end{rmk}
\subsection{The cases \texorpdfstring{$\bm{\iso\left(\mathbf{T}^c\right)=\g_2}$, $\bm{\su\left(2\right)\op\su_c\left(2\right)}$,
$\bm{\R\op\su_c\left(2\right)}$}{iso(T^c)=g_2, su(2)+su_c(2), R+su_c(2)}} 
The vector field $e_8$, the spi\-nor fields $\Psi_1-\Psi_2$, $\Psi_9-\Psi_{10}$ and the globally defined $3$-form
$\left(Z\wedge e_7+D\right)$ are parallel with respect to $\nabla^c$. The tangent bundle $TM^8$ of $\left(M^8,g,\Phi\right)$
splits into two components preserved by $\nabla^c$,
\be
TM^8=E\op \R\cdot e_8.
\ee
The torsion form $\T^c$ satisfies
\be
e_8\hook \T^c=b_2\,\left(e_{12}-e_{34}\right)
\ee
and the real parameter $b_2$ vanishes if and only if $\iso\left(\T^c\right)=\g_2$ or $\su\left(2\right)\op\su_c\left(2\right)$
(see section \ref{sec:5}). Consequently, we split the discussion into $\iso\left(\T^c\right)=\g_2$ or $\su\left(2\right)\op
\su_c\left(2\right)$ and $\iso\left(\T^c\right)=\R\op\su_c\left(2\right)$.

Suppose $\iso\left(\T^c\right)=\g_2$ or $\su\left(2\right)\op\su_c\left(2\right)$. Here the torsion form does not
depend on $e_8$, and $\nabla^ge_8=0$. Therefore a complete and simply connected $M^8$ is isometric to the Riemannian
product of a $7$-manifold $Y^7$ with $\R$. The $3$-form $\T^c$ is contained in the component $\Lambda^3_1\left(
\R^7\right)\op\Lambda^3_{27}\left(\R^7\right)$ of the decomposition of $3$-forms on $\R^7$ into $\G_2$-irreducible
components (see \cite{FI02}). Consequently, the space $Y^7$ is isometric to a cocalibrated $\G_2$-manifold with fundamental
form $\left(Z\wedge e_7+D\right)$ and parallel characteristic torsion $\T^c$. The corresponding characteristic holonomy
of $Y^7$ is the subalgebra $\hol\left(\nabla^c\right)$ of $\iso\left(\T^c\right)$. Finally, we can reconstruct the
considered $\Spin\left(7\right)$-structure using equation (\ref{eqn:1}).
\begin{thm}
Let $\left(M^8,g,\Phi\right)$ be a complete, simply connected $\Spin\left(7\right)$-manifold with parallel characteristic
torsion $\T^c$ and $\iso\left(\T^c\right)=\g_2$ or $\su\left(2\right)\op\su_c\left(2\right)$. Then $M^8$ is isometric to
the Riemannian product of $\R$ with a cocalibrated $\G_2$-manifold with parallel characteristic torsion and characteristic
holonomy contained in $\iso\left(\T^c\right)$.
\end{thm}
Those $Y^7$ with non-Abelian characteristic holonomy were extensively studied in \cite{Fri06}. We provide an
example for $Y^7$ with Abelian characteristic holonomy. This restricts necessarily to $\iso\left(\T^c\right)=
\su\left(2\right)\op\su_c\left(2\right)$.
\begin{exa}\label{exa:1}
If $\hol\left(\nabla^c\right)=\mfr{0}$ the torsion parameters satisfy $b_2=0$ and $b_1=-a_1$. Moreover, the spinor
fields $\Psi_1,\ldots,\Psi_{16}$ are $\nabla^c$-parallel, the characteristic curvature tensor vanishes $\mrm{R}^c=0$
and the torsion form $\T^c$ is proportional to $Z_2\wedge e_7=e_{567}$. In particular, the latter implies that
$\T^c$ does not depend on the $\nabla^c$-parallel vector fields $e_1$, $e_2$, $e_3$, $e_4$ and $e_8$. Computing
the Lie bracket $\left[\,\cdot\,,\,\cdot\,\right]$ of the Lie algebra corresponding to $M^8$ via $\T^c\left(X,Y,Z
\right)=-g\left(\left[X,Y\right],Z\right)$ results in the conclusion that a complete and simply connected $M^8$ is
isometric to the Riemannian product $\R^5\x\SU\left(2\right)$.
\end{exa}
We proceed with $\iso\left(\T^c\right)=\R\op\su_c\left(2\right)$. With the algebraic considerations of \autoref{sec:5}
we immediately obtain the following result:
\begin{prop}
Any $\Spin\left(7\right)$-manifold with parallel characteristic torsion $\T^c$ and $\iso\left(\T^c\right)=\R\op\su_c
\left(2\right)$ is isometric to a homogeneous naturally reductive space.
\end{prop}
\begin{exa}\label{exa:2}
Let $\hol\left(\nabla^c\right)=\mfr{0}$. Then the spinor fields $\Psi_1,\ldots,\Psi_{16}$ are $\nabla^c$-parallel,
$\mrm{R}^c=0$ and the characteristic torsion is proportional to one of the following two $3$-forms:
\be
\alpha_\pm=\left(Z_1-2\,Z_2\right)\wedge e_7 + D \pm \sqrt{3}\cdot\left(Z_3\wedge e_8\right).
\ee
Each of these $3$-forms can be reconstructed with
\be
\alpha_\pm=-g\left(\left[X,Y\right],Z\right)
\ee
using the bracket $\left[\,\cdot\,,\,\cdot\,\right]$ of the respective Lie algebra
\be
\su\left(3\right)=\mrm{span}\left(P_4,P_3,P_1,P_2,-P_5,-P_6,P_7,\pm{\sst\sqrt{1/3}}\,\left(P_7+2\,P_8\right)\right).
\ee
We conclude that the considered $\Spin\left(7\right)$-manifold $M^8$ is isometric to $\SU\left(3\right)$.
\end{exa}
\subsection{The case \texorpdfstring{$\bm{\iso\left(\mathbf{T}^c\right)=\R\op\su\left(2\right)}$}{iso(T^c)=R+su(2)}}
The characteristic holonomy $\hol\left(\nabla^c\right)$ is necessarily equal to $\R\op\su\left(2\right)$. A $\Spin\left(
7\right)$-structure $\left(M^8,g,\Phi\right)$ with non-trivial parallel characteristic torsion and $\hol\left(\nabla^c
\right)=\R\op\su\left(2\right)=\iso\left(\T^c\right)$ admits two $\nabla^c$-parallel $2$-forms $Z_1=e_{12}+e_{34}$ and
$Z_2=e_{56}$, two $\nabla^c$-parallel spinor fields $\Psi_9$ and $\Psi_{10}$, and the tangent bundle $TM^8$ splits into
the sum of two $\R\op\su\left(2\right)$-invariant subbundles,
\be
TM^8=E_1\op E_2.
\ee
Here $E_2$ is spanned by $\left\{e_7,e_8\right\}$. The torsion form $\T^c$ does not depend on $e_7$ and $e_8$, and
therefore a complete, simply connected $M^8$ is isometric to the Riemannian product of a $2$-dimensional manifold with
a $6$-dimensional manifold $X^6$. The space $X^6$ is isometric to an almost Hermitian manifold with K\"ahler form
$Z=Z_1+Z_2$ and non-trivial parallel characteristic torsion $\T^c$. This torsion form is contained in the $\Lambda^3_{12}
\left(\R^6\right)$-component in the decomposition of $\Lambda^3\left(\R^6\right)$ under the action of $\Un\left(3\right)$
(see \cite{AFS05}). Analyzing the representation of $\Hol\left(\nabla^c\right)\subset\Un\left(1\right)\x\Un\left(2\right)
\subset\Un\left(3\right)$ on $\R^6\cong E_1$, we conclude with \cites{FKMS97} that $X^6$ carries the structure of a twistor
space and is homothetic to either $\C\Pro^3$ or $\Flag\left(1,2\right)$. The representation of $\Hol\left(\nabla^c\right)$
on $\R^2\cong E_2$ defines a non-trivial rotation. The $2$-dimensional component is consequently isometric to $\mrm{S}^2$.
We can reconstruct the considered $\Spin\left(7\right)$-structure from the Hermitian structure of its components,
\be
\Phi=\frac{1}{2}\cdot \omega\wedge\omega + \mrm{Re}\left(F\right),
\ee
where
\begin{align*}
\omega&:=Z+e_7\wedge e_8, & F&:=\left(e_1+i\,e_2\right)\wedge\left(e_3+i\,e_4\right)\wedge\left(e_5+i\,e_6\right)\wedge\left(
e_7+i\,e_8\right).
\end{align*}
Finally we obtain the following result.
\begin{thm}\label{thm:3}
A complete, simply connected $\Spin\left(7\right)$-manifold with parallel characteristic torsion $\T^c$ and $\iso\left(\T^c
\right)=\R\op\su\left(2\right)$ is isometric to the Riemannian product of a $2$-sphere with either the projective space $\C
\Pro^3$ or the flag manifold $\Flag\left(1,2\right)$, both equipped with their standard nearly K\"ahler structure from the
twistor construction.
\end{thm}
\subsection{The case \texorpdfstring{$\bm{\iso\left(\mathbf{T}^c\right)=\un\left(2\right)}$}{iso(T^c)=u(2)}}
The differential forms $Z_1$, $Z_2$, $D$, the vector fields $e_7$, $e_8$ and the spinor fields $\Psi_1$, $\Psi_2$,
$\Psi_9$, $\Psi_{10}$ are parallel with respect to $\nabla^c$ and the tangent bundle of $\left(M^8,g,\Phi\right)$
splits into the following $\un\left(2\right)$-invariant components:
\be
TM^8=E_1 \op E_2\op\R\cdot e_7\op\R\cdot e_8.
\ee
The subbundle $E_2$ is spanned by $e_5$ and $e_6$.

There are two types of characteristic torsion. We start to discuss the case of torsion type $I$. Setting
\begin{align*}
TM_+&=E_1\op\left(\left(a_2+b_1\right)\cdot e_7+a_1\cdot e_8\right), & TM_-&=E_2\op\left(\left(a_2-\frac{2}{5}\,b_1\right)
\cdot e_7+a_1\cdot e_8\right)
\end{align*}
and $\T=\T_++\T_-=\T^c$ satisfies system (\ref{eqn:3}) and
\begin{align*}
\T_+&=Z_1\wedge\left(\left(a_2+b_1\right)\cdot e_7+a_1\cdot e_8\right), & \T_-&=5\cdot Z_2\wedge\left(\left(a_2-\frac{2}{5}\,
b_1\right)\cdot e_7+a_1\cdot e_8\right).
\end{align*}
The equation $e_7\hook \T^c=0$ or $e_8\hook \T^c=0$ holds, if $a_2=b_1=0$ or $a_1=0$ respectively. Consequently, a
simply connected and complete $M^8$ is isometric to the Riemannian product of a $5$-manifold with a $3$-manifold each
carrying a Sasakian structure. The respective fundamental forms are $Z_1$ and $Z_2$.
\begin{thm}
Let $\left(M^8,g,\Phi\right)$ be a complete, simply connected $\Spin\left(7\right)$-manifold with parallel characteristic
torsion $\T^c$ and $\iso\left(\T^c\right)=\un\left(2\right)$. Suppose that the torsion form is of type $I$. Then $M^8$ is
isometric to the Riemannian product of a Sasakian $3$-manifold with a $5$-dimensional Sasakian structure.
\end{thm}
We proceed with the discussion of torsion type $I\!I$. Solving the equation $X\hook \T^c=0$ leads to $X=a_1\cdot
e_7-\left(a_2+b_1\right)\cdot e_8\neq0$ (see \autoref{sec:5}). Consequently, a complete, simply connected $M^8$ is
isometric to the Riemannian product of a $7$-dimensional integrable $\G_2$-manifold $Y^7$ with $\R$ (see \cite{FI03}).
\begin{thm}
Let $\left(M^8,g,\Phi\right)$ be a complete, simply connected $\Spin\left(7\right)$-manifold with parallel characteristic
torsion $\T^c$ and $\iso\left(\T^c\right)=\un\left(2\right)$. Suppose that the torsion form is of type $I\!I$. Then $M^8$
is isometric to the Riemannian product of $\R$ with an integrable $\G_2$-manifold with parallel characteristic torsion and
characteristic holonomy contained in $\un\left(2\right)$.
\end{thm}
\begin{thank}
We wish to thank Nils Sch\"omann and Richard Cleyton for discussions and the SFB 647:~\emph{Space--Time--Matter} for financial
support.
\end{thank}
%
%
%
%

%
%
%

\begin{bibdiv}
%
\begin{biblist}
%
\bib{AF04}{article}{
  author={Agricola, I.},
  author={Friedrich, T.},
  title={The Casimir operator of a metric
         connection with totally skew-symmetric torsion},
  journal={Journ. Geom. Phys.},
  volume={50},
  date={2004},
  pages={188--204} }
%
\bib{AFNP05}{article}{
  author={Agricola, I.},
  author={Friedrich, T.},
  author={Nagy, P.-A.},
  author={Puhle, C.},
  title={On the Ricci tensor in the common sector of type II string theory},
  journal={Class. Quant. Grav.},
  volume={22},
  date={2005},
  pages={2569--2577},
}
%
\bib{AFS05}{article}{
  author={Alexandrov, B.},
  author={Friedrich, T.},
  author={Schoemann, N.},
  title={Almost Hermitian {$6$}-Manifolds Revisited},
  journal={Journ. Geom. Phys.},
  volume={53},
  date={2005},
  pages={1--30},
}
%
\bib{BGM04}{article}{
  author={Boyer, C. P.},
  author={Galicki, K.},
  author={Matzeu, P.},
  title={On eta-Einstein sasakian geometry},
  journal={Comm. Math. Phys.},
  volume={262},
  date={2006},
  pages={177--208},
}
%
\bib{Cab95}{article}{
  author={Cabrera, F. M.},
  title={On Riemannian manifolds with {$\Spin\left(7\right)$}-structure},
  journal={Publ. Math.},
  volume={46},
  date={1995},
  pages={271--283},
}
%
\bib{CM08}{article}{
  author={Cleyton, R.},
  author={Moroianu, A.},
  title={Connections with parallel torsion in Riemannian geometry},
  pages={to appear, can be obtained by the authors},
}
%
\bib{CS04}{article}{
  author={Cleyton, R.},
  author={Swann, A.},
  title={Einstein Metrics via Intrinsic or Parallel Torsion},
  journal={Math. Z.},
  volume={247},
  date={2004},
  pages={513--528},
}
%
\bib{Dyn57a}{article}{
  author={Dynkin, E. B.},
  title={Semisimple subalgebras of semisimple Lie algebras},
  journal={Am. Math. Soc., Transl., II. Ser.},
  volume={6},
  date={1957},
  pages={111--243},
}
%
\bib{Dyn57b}{article}{
  author={Dynkin, E. B.},
  title={Maximal subgroups of the classical groups},
  journal={Am. Math. Soc., Transl., II. Ser.},
  volume={6},
  date={1957},
  pages={245--378},
}
%
\bib{Fer86}{article}{
  author={Fern\'andez, M.},
  title={A classification of Riemannian manifolds with structure group {$\Spin\left(7\right)$}},
  journal={Ann. Mat. Pura Appl., IV. Ser.},
  volume={143},
  date={1986},
  pages={101--122},
}
%
\bib{Fri06}{article}{
  author={Friedrich, T.},
  title={{$\G_2$}-manifolds with parallel characteristic torsion},
  journal={Diff. Geom. Appl.},
  volume={25},
  date={2007},
  pages={632--648},
}
%
\bib{FI02}{article}{
  author={Friedrich, T.},
  author={Ivanov, S.},
  title={Parallel spinors and connections with skew-symmetric torsion in string theory},
  journal={Asian Journ. Math},
  volume={6},
  date={2002},
  pages={303--336},
}
%
\bib{FI03}{article}{
  author={Friedrich, T.},
  author={Ivanov, S.},
  title={Killing spinor equations in dimension $7$ and geometry of integrable $\G_2$-manifolds},
  journal={Journ. Geom. Phys.},
  volume={48},
  date={2003},
  pages={1--11},
}
%
\bib{FKMS97}{article}{
  author={Friedrich, T.},
  author={Kath, I.},
  author={Moroianu, A.},
  author={Semmelmann, U.},
  title={On nearly parallel {$\G_2$}-structures},
  journal={Journ. Geom. Phys.},
  volume={23},
  date={1997},
  pages={256--286},
}
%
\bib{FK00}{article}{
  author={Friedrich, T.},
  author={Kim, E. C.},
  title={The Einstein-Dirac equation on Riemannian
         spin manifolds},
  journal={Journ. Geom. Phys.},
  volume={33},
  date={2000},
  pages={128--172},
}
%
\bib{HL82}{article}{
  author={Harvey, R.},
  author={Lawson, H. B.},
  title={Calibrated geometries},
  journal={Acta Math.},
  volume={148},
  date={1982},
  pages={47--157},
}
%
\bib{Iva04}{article}{
  author={Ivanov, S.},
  title={Connections with torsion, parallel spinors and geometry of {$\Spin\left(7\right)$}-manifolds},
  journal={Math. Res. Lett.},
  volume={11},
  date={2004},
  pages={171--186},
}
%
\bib{Puh07}{article}{
  author={Puhle, C.},
  title={The Killing spinor equation with higher order potentials},
  eprint={http://lanl.arxiv.org/abs/0707.2217},
  date={2007},
  pages={to appear in Journ. Geom. Phys.},
}
%
\bib{Str86}{article}{
  author={Strominger, A.},
  title={Superstrings with torsion},
  journal={Nucl. Phys. B},
  volume={274},
  date={1986},
  pages={253--284},
}
%
\bib{McKW89}{article}{
  author={McKenzie Wang, Y.},
  title={Parallel spinors and parallel forms},
  journal={Ann. Global Anal. Geom.},
  volume={7},
  date={1989},
  pages={59--68},
}
%
\bib{Sch07}{article}{
  author={Schoemann, N.},
  title={Almost hermitian structures with parallel torsion},
  journal={J. Geom. Phys.},
  volume={57},
  date={2007},
  pages={2187--2212}, 
}
%
\end{biblist}
%
\end{bibdiv}
\end{document}